\documentclass[12pt]{amsart}
\usepackage{amssymb,verbatim,enumerate,ifthen}
\usepackage[mathscr]{eucal}
\usepackage[utf8]{inputenc}
\usepackage[T1]{fontenc}
\usepackage{enumitem}
\makeatletter
\@namedef{subjclassname@2010}{%
  \textup{2010} Mathematics Subject Classification}
\makeatother

\newtheorem{thm}{Theorem}
\newtheorem{cor}[thm]{Corollary}
\newtheorem{prop}[thm]{Proposition}
\newtheorem*{prop*}{Proposition}
\newtheorem{lem}[thm]{Lemma}

\newtheorem*{exa}{Example}

\newcommand\nolabel[1]{\nonumber}

\newcommand\R{\mathbb{R}}
\renewcommand\P{\mathcal{P}}
\newcommand\N{\mathbb{N}}
\newcommand\Q{\mathbb{Q}}
\newcommand\Z{\mathbb{Z}}
\newcommand{\QA}[1]{A^{[#1]}}
\newcommand{\calM}{\mathcal{M}}
\newcommand{\calF}{\mathcal{F}}
\newcommand{\calI}{\mathcal{I}}

\newcommand{\intgen}[1]{\left[\left[#1\right]\right]}

\DeclareMathOperator{\interior}{int}

\def\eq#1{{\rm(\ref{#1})}}
\def\Eq#1#2{\ifthenelse{\equal{#1}{*}}
  {\begin{equation*}\begin{aligned}[]#2\end{aligned}\end{equation*}}
  {\begin{equation}\begin{aligned}[]\label{#1}#2\end{aligned}\end{equation}}}


\title{Interval-type theorems concerning quasi-arithmetic means}
\author[P. Pasteczka]{Pawe\l{} Pasteczka}
\address{Institute of Mathematics \\ Pedagogical University of Cracow \\ Podchor\k{a}\.zych str. 2, 30-084 Krak\'ow, Poland}
\email{pawel.pasteczka@up.krakow.pl}

\begin{document}
\maketitle

\begin{abstract}
Family of quasi-arithmetic means has a natural, partial order (point-wise order) $A^{[f]}\le A^{[g]}$  if and only if $A^{[f]}(v)\le A^{[g]}(v)$ for all admissible vectors $v$ ($f,\,g$ and, later, $h$ are continuous and monotone and defined on a common interval).

Therefore one can introduce the notion of interval-type sets (sets $\mathcal{I}$ such that whenever $A^{[f]} \le A^{[h]} \le A^{[g]}$ for some $A^{[f]},\,A^{[g]} \in \mathcal{I}$ then $A^{[h]} \in \mathcal{I}$ too). 

Our aim is to give examples of interval-type sets involving vary smoothness assumptions of generating functions. 
\end{abstract}

%{\Huge Praca wysylana jako II w kolenosci (z kanapkowych), przerobic Introduction itd.}
%Trivial: bigger/smaller than some fixed function (not necessary mean), intersection of any number of interval-type sets.

%Let me begin this paper at the brief, however informal, description of our main purpose.
\section{Introduction}
In a recent paper \cite{Pas17a} author introduced a new definition concerning means. A family $\calM$ of means (functions) defined on a common domain is embedded in a natural partial order, that is for every $M,N \in \calM$ we have 
\Eq{*}{
M \le N \iff M(x) \le N(x) \text{ for all }x.
}

In this setting we call $\calI \subset \calM$ to be an interval-type set in $\calM$ (briefly: interval-type set or interval) if whenever $P \in \calM$ and $M\le P \le N$ for some $M,\,N \in \calI$ then also $P \in \calI$.
  
Many families of means are linearly ordered by this process. For example one of them most classical result in a theory of means states that power means are linearly ordered, that is if we denote by $\P_p$ the $p$-th power mean, then $(\{\P_p\}_{p \in \R},\le)$ is isomorphic to $(\R,\le)$ under the natural isomorphism $\P_p \mapsto p$. In particular all intervals in this family could be trivially described.

Situation becomes much more interesting if there appear means which are not comparable among each other.
Perhaps the most famous family of this type are quasi-arithmetic means.
They were introduced in series of nearly simultaneous papers in a beginning of 1930s \cite{Def31,Kol30,Nag30} as a generalization of already mentioned family of power means. 
For a continuous and strictly monotone function $f \colon I \to \R$ ($I$ is an interval) and a vector $a=(a_1,a_2,\dots,a_n) \in I^n$, $n \in \N$ we define 
\Eq{*}{
\QA{f}:=f^{-1}\left( \frac{f(a_1)+f(a_2)+\cdots+f(a_n)}{n} \right).
}

It is easy to verify that for
$I=\R_+$ and $f=\pi_p$, where $\pi_p(x):=x^p$ if $p\ne 0$ and $\pi_0(x):=\ln x$, then mean $\QA{f}$ coincides 
with $\P_p$ (this fact had been already noticed by Knopp \cite{Kno28} before quasi-arithmetic means were formally introduced).

Then, whenever $f$ and $g$ are defined on a common interval $I$, we get
\Eq{*}{
\QA{f} \le \QA{g} \text{ if and only if } \QA{f}(a) \le \QA{g}(a) \text{ for all }a \in \bigcup_{n=1}^{\infty} I^n.
}
As we do not define comparability of means defined on two different intervals, throughout all quasi-arithmetic means are considered on an arbitrary, but common, interval (from now on denoted by $I$). We will be dealing with interval-type sets in a family of quasi-arithmetic means defined on $I$ (we will call them briefly interval-type sets or intervals).

Let us recall some simple, however important, results from our previous paper \cite{Pas17a}. It could be proved that interval-type sets inherit many properties of regular intervals in $\R$. For example intersection of any number of intervals are again an interval, increasing sum of intervals are again an interval and so on -- proofs of this facts are elementary and omitted here; for detailed discussion we refer the reader to \cite{Pas17a}.
Moreover, if $D \subset \bigcup_{n=1}^{\infty} I^n$ and 
$L,\,U \colon D \to \R$ are arbitrary functions and then both
\Eq{*}{
[L,+\infty):=\{\QA{f} \colon L(v) \le \QA{f}(v) \text{ for all } v \in D\}, \\
(L,+\infty):=\{\QA{f} \colon L(v) < \QA{f}(v) \text{ for all } v \in D\}
}
are intervals. Similarly we can define all possible intervals of this type involving $-\infty$. Having this we define bounded intervals of this type as an intersection; for example $[L,U):=[L,+\infty) \cap (-\infty,U)$ etc. 

Furthermore, as we have only a partial order, it is reasonable to define, for every family $\calF$ of quasi-arithmetic means, the smallest interval-type set containing $\calF$. We will denote such a set by 
\Eq{*}{
\intgen{\calF}:=\bigcap \big\{\mathcal{I} \colon \mathcal{I} \text{ is an interval and } \calF \subset \mathcal{I}\big\}.
}
It has also a natural upward (equivalent) definition
\Eq{E:upwardDef}{
\intgen{\calF}=\bigcup_{{X,Y \in \calF}\atop {X \le Y}} [X,Y].
}
Proof of this equality is elementary and we omit it.

\section{Comparability among quasi-arithmetic means} 
It could happen that two intervals has a non-empty intersection although its sum is not an interval. Indeed, the family
\Eq{*}{
[\QA{f}]^{\ast}:=(-\infty,\QA{f}] \cup [\QA{f},+\infty)
}
is a family of all quasi-arithmetic means which are comparable with $\QA{f}$.
Investigating properties of this set is somehow outside the scope of the present paper, as it is not an interval. Let us just notice that for arithmetic mean $[\QA{x}]^{\ast}$ is a quasi-arithmetic means generated by either convex functions or concave functions, which is the classical application of Jensen inequality.

In fact Jensen inequality is closely related with comparability of quasi-arithmetic means. In what follows we will present a number of equivalent conditions in a series of propositions. They will be uniquely numerated, as we will refer to each of them just by mentioning its identifier.

\begin{prop*}
 Let $f,g \colon I \to \R$  be a continuous and monotone functions. Then $\QA{f} \le \QA{g}$ if and only if 
 \begin{enumerate}[label={\roman*.},series=theoremconditions]
\item \label{CC:gf^-1} $g$ is increasing and $g \circ f^{-1}$ is convex or $g$ is decreasing and $g \circ f^{-1}$ is concave,
\item \label{CC:fg^-1} $f$ is increasing and $f \circ g^{-1}$ is concave or $f$ is decreasing and $f \circ g^{-1}$ is convex.
\end{enumerate}
\end{prop*}

In fact this proposition possess a lot of symmetries as we have the well-known equality condition (cf. \cite{Nag30})
\Eq{E:Eqcond}{
\QA{f} &= \QA{g} \iff  
\left( {\text{ there exists }\alpha,\,\beta \in \R\text{ with }\alpha\ne0} \atop { \text{ such that }f=\alpha \cdot g+ \beta} \right).
}

It is easy to observe that $g \circ f^{-1}$ is continuous, so its convexity, $t$-convexity for given $t \in [0,1]$, Jensen convexity ($1/2$-convexity) are all equivalent.
Therefore we obtain a number of conditions which provide comparability of quasi-arithmetic means. This is a folk result in a theory of means
\begin{prop*}
  Let $f,g \colon I \to \R$  be a continuous and monotone functions. Then the following conditions are equivalent to $\QA{f} \le \QA{g}$
 \begin{enumerate}[resume*=theoremconditions]
   \item \label{CC:I^omega} $\QA{f}(a) \le \QA{g}(a)$ for all $a \in \bigcup_{n=1}^{\infty} I^n$;
   \item \label{CC:I^k} $\QA{f}(a) \le \QA{g}(a)$ for some $k \in \N$ and all $a \in I^k$;
   \item \label{CC:I^2weighted-givenxi} $\QA{f}_\xi(a) \le \QA{g}_\xi(a)$ for some $\xi \in (0,1)$ and all $a \in I^2$, where $\QA{f}_\xi(a):=f^{-1}(\xi f(a_1)+(1-\xi)f(a_2))$;
   \item \label{CC:I^2weighted-allxi} $\QA{f}_\xi(a) \le \QA{g}_\xi(a)$ for all $\xi \in (0,1)$ and all $a \in I^2$.
  \end{enumerate}
\end{prop*}

Additionally we have a condition, that were given by P\'ales \cite{Pal91} (see also \cite{Pas16b})
{\it
 \begin{enumerate}[resume*=theoremconditions]
   \item \label{CC:Pales} $\frac{f(x)-f(y)}{f(x)-f(z)} \ge \frac{g(x)-g(y)}{g(x)-g(z)} $ for all $x,\,y,\,z \in I$, $x<y<z$.
  \end{enumerate}
}
Furthermore, we know that a differentiable function is convex/concave if and only if its derivative is non-decreasing/non-increasing. Applying this we get next comparability conditions.

\begin{prop*}
  Let $f,g \colon I \to \R$  be a monotone and differentiable functions with $f'\cdot g' \ne 0$. Then $\QA{f} \le \QA{g}$
if and only if either
\begin{enumerate}[resume*=theoremconditions]
   \item \label{CC:f'/g'-samemonotonicity} $f$ and $g$ are of the same monotonicity (both increasing or both decreasing) and $f'(x)/g'(x)$ is non-increasing (equivalently $g'(x)/f'(x)$ is non-decreasing);
   \item \label{CC:f'/g'-conversemonotonicity} $f$ and $g$ are of the converse monotonicity (one increasing, second decreasing) and $f'(x)/g'(x)$ is non-decreasing (equivalently $g'(x)/f'(x)$ is non-increasing).
  \end{enumerate}
\end{prop*}

Now we turn into the result of Mikusi\'nski \cite{Mik48}. He, and independently \L{}ojasiewicz (compare \cite[footnote 2]{Mik48}), expressed handy tool to compare quasi-arithmetic means in terms of operator $f \mapsto f''/f'$ (the negative of this operator is used to be called an Arrow-Pratt index). More precisely their result reads
\begin{prop*}
Let $I$ be an interval, $f,\,g \in \mathcal{C}^{2}(I)$, $f' \cdot g' \ne 0$ on $I$. 
Then $\QA{f} \le \QA{g}$ if and only if
\begin{enumerate}[resume*=theoremconditions]
   \item \label{CC:Mikusinski} $\frac{f''(x)}{f'(x)} \le \frac{g''(x)}{g'(x)}$ for all $x \in I$.
  \end{enumerate}
\end{prop*}
Using this result we immediately obtain some ``Mikusi\'nski-type intervals''
\Eq{*}{
\tilde{\mathcal{M}}(x_0,U):=\left\{\QA{f} \colon {{f \text{ is twice continuously differentiable in some}} \atop
{\text{neighborhood of }x_0,\,f'(x_0) \ne 0,\,\frac{f''(x_0)}{f'(x_0)}\in U}}\right\},
}
where $x_0 \in I$ and $U \subset \R$ is an interval.
Nevertheless $\tilde{\mathcal{M}}(x_0,U)$ is usually not an interval-type set, therefore we extend this set to an interval in the way that was described in the introduction
\Eq{*}{
\mathcal{M}(x_0,U):=\intgen{\tilde{\mathcal{M}}(x_0,U)}.
}

This lead us to the following problem. A family $\tilde{\mathcal{M}}(x_0,U)$ contains only $\mathcal{C}^2$ functions around $x_0$ with $f'(x_0) \ne 0$.
But what about $\mathcal{M}(x_0,U)$?

By \eq{E:upwardDef} we know that for all $\QA{h} \in \mathcal{M}(x_0,U)$ we have 
\Eq{E:sandwich}{
\QA{f}\le \QA{h}\le \QA{g} 
}
for some $f,\,g \in \mathcal{C}^2(V)$, $f'\cdot g' \ne 0$ and open interval $V \ni x_0$.

In fact it does not imply that the second derivative of $h$ at $x_0$ exists, which can be illustrated in a simple example

\begin{exa}
 Let $I=(0,+\infty)$, $f(x)=x$, $g(x)=x^2$, $x_0=1$,
\Eq{*}{
h(x)=\begin{cases} 
            x & x \in (0,1) \\
	    \frac{x^2+1}2 & x \in (1,2)
     \end{cases}
}
Then, by \ref{CC:f'/g'-samemonotonicity}, assertion \eq{E:sandwich} holds but $h$ is $\mathcal{C}^1$ only.
\end{exa}

Despiting this drawback, it can be proved that if $f,\,g \in \mathcal{C}^2(I)$ with nonvanishing derivative and \eq{E:sandwich} holds, then $h$ is continuously differentiable for all $x \in U$ and also $h'$ is nowhere vanishing. 

Nevertheless to obtain an interval-type set assumption on $f$, $g$, and $h$ have to be the same. Thus we want to prove that if $f$ and $g$ are continuously differentiable with nonvanishing derivative, then so is $h$. Equivalently, family of quasi-arithmetic means generated by $\mathcal{C}^1$ functions with nonvanishing derivative is an interval (it will be done in Theorem~\ref{thm:dernon_van}). %It is a natural generalization of the fact mentioned in the previous paragraph.

\section{Interval-type sets in a family quasi-arithmetic means} 
In the following section we will prove a number of examples of interval-type sets involving vary smoothness assumptions of generating functions. 
Let us first prove some abstract theorem.

\begin{thm}
\label{thm:Error}
Let $I$ be a compact interval, $x_0\in I$, $f_0 \colon I \to \R$ with $f_0(x_0)=0$. Let $\calF \subset \mathcal{C}(I)$ be an interval
such that 
%$\alpha \calF =\calF$ for all $\alpha > 0$ and 
$\calF \subseteq o(f_0)$ in a right/left neighborhood of $x_0$.
Then the family 
\Eq{*}{
\QA{f_0+\calF} :=\{\QA{f_0+f} \colon f \in \calF \text{ and } f_0+f \text{ is strictly monotone on }I\}.
}
is an interval.
\end{thm}

\begin{proof}
We want to bind the cases where $x_0$ is in the interior of the interval and is the endpoint. In the proof we will concern right neighborhood of the point, 
therefore we have $x_0 \ne \sup I$. Second case is completely analogous. Similarly assume that $f_0$ is increasing.

Take any $x_1 \in I$ such that $x_1>x_0$. 
Let $r_1,r_2 \in \calF$ and $f:=f_0+r_1$, $g:=f_0+r_2$. By the definition there holds $f(x_0)=g(x_0)=0$.
Denote
\Eq{*}{
\hat f(x):=\frac{f(x)-f(x_0)}{f(x_1)-f(x_0)}=\frac{f(x)}{f(x_1)}, \quad \hat g(x):=\frac{g(x)-g(x_0)}{g(x_1)-g(x_0)}=\frac{g(x)}{g(x_1)}.
}
Let us consider an arbitrary function $\tilde h \colon I \to \R$ satisfying 
\Eq{*}{
\QA{f} \le \QA{\tilde h} \le \QA{g}.
}
By \eqref{E:Eqcond}, there exist a unique function $h$ such that $h(x_0)=0$, $h(x_1)=1$, and $\QA{\tilde h}=\QA{h}$. 

Then, by \ref{CC:Pales}, we get $\hat g(x) \le h(x) \le \hat f(x)$ for all $x \in (x_0,x_1)$.
Thus
\Eq{*}{
\frac{g(x)}{g(x_1)} \le h(x) \le \frac{f(x)}{f(x_1)}, \quad x \in (x_0,x_1).
}
Therefore
\Eq{*}{
\frac{f_0(x)+r_2(x)}{g(x_1)} \le h(x) \le \frac{f_0(x)+r_1(x)}{f(x_1)}, \quad x \in (x_0,x_1).
}
%For the sake of brevity let us denote briefly $k(x) + \calF \le l(x)$ if there exists $r \in \calF$ such that $k(x)+r(x)\le l(x)$ for $x\in (x_0,x_1)$.
%Having this notation we get
Then
\Eq{E_FF}{
\frac{f_0(x)}{g(x_1)}+\frac{1}{g(x_1)} \cdot r_2(x) \le h(x) \le \frac{f_0(x)}{f(x_1)} +\frac{1}{f(x_1)} \cdot r_1(x)
}
It implies
\Eq{*}{
\left( \frac{1}{g(x_1)}-\frac{1}{f(x_1)}\right) f_0(x) \le \frac{1}{f(x_1)}\cdot r_1(x)-\frac{1}{g(x_1)}\cdot r_2(x)
}

But $\frac{1}{f(x_1)}r_1(x)-\frac{1}{g(x_1)}r_2(x) \in o(f_0)$ in a right neighborhood of $x_0$. 
Thus 
\Eq{*}{
\frac{1}{g(x_1)}-\frac{1}{f(x_1)}\le 0.
}
Consequently $g(x_1)\ge f(x_1)$. As $x_1$ was an arbitrary number greater than $x_0$ we obtain
$g(x) \ge f(x)$ for $x>x_0$. Thus
\Eq{r12}{
r_2(x) \ge r_1(x) \text{ for }x>x_0.
}

Now observe that
\Eq{*}{
\frac{f(x)}{f(x_1)} \le h(x) &\le \frac{g(x)}{g(x_1)} \text{ for }x>x_1. \\
%\frac{f_0(x)+r_1(x)}{f(x_1)}  \ge h(x) &\ge \frac{f_0(x)+r_2(x)}{g(x_1)} \text{ for }x>x_1.\\
\frac{f_0(x)+r_1(x)}{f_0(x_1)+r_1(x_1)}  \le h(x) &\le \frac{f_0(x)+r_2(x)}{f_0(x_1)+r_2(x_1)} \text{ for }x>x_1.
}
%As $x_1>x_0$, \eq{r12} is valid and both $f(x_1)$ and $g(x_1)$ are positive. Therefore we obtain, for $x>x_1$,
%\frac{1}{f(x_1)}  \ge \frac{h(x)}{f_0(x)+r_1(x)} &\ge \frac{1}{g(x_1)} \text{ for }x>x_1.\\
%\frac{f_0(x)}{f(x_1)}  \ge h(x) \cdot (1-\frac{r_1(x)}{f_0(x)+r_1(x)}) &\ge \frac{f_0(x)}{g(x_1)} \text{ for }x>x_1.\\
%\Eq{*}{
%\frac{f_0(x)+r_2(x)}{f(x_1)} \ge \frac{f_0(x)+r_1(x)}{f(x_1)} \ge h(x) \ge \frac{f_0(x)+r_2(x)}{g(x_1)} \ge \frac{f_0(x)+r_1(x)}{g(x_1)}.
%}
%Now we can add both side and denote $r_{3/2}:=\tfrac 12 (r_1+r_2)$ and obtain
%\Eq{*}{
%\frac{f_0(x)+r_{3/2}(x)}{f(x_1)}  \ge h(x) &\ge \frac{f_0(x)+r_{3/2}(x)}{g(x_1)} \text{ for }x>x_1.
%}
%Thus
%\Eq{*}{
%\frac1{f(x_1)} \ge \frac{h(x)}{f_0(x)+r_{3/2}(x)} \ge \frac1{g(x_1)} \text{ for }x>x_1.\\
%\frac1{f_0(x_1)+r_1(x_1)} \ge \frac{h(x)}{f_0(x)+r_{3/2}(x)} \ge \frac1{f_0(x_1)+r_2(x_1)} \text{ for }x>x_1.\\
%\frac{f_0(x)+r_{3/2}(x)}{f_0(x_1)+r_1(x_1)} \ge h(x) \ge \frac{f_0(x)+r_{3/2}(x)}{f_0(x_1)+r_2(x_1)} \text{ for }x>x_1.\\
%}
Denote
\Eq{*}{
r_\beta=(\beta-1)r_2+(2-\beta)r_1 \in \calF,\quad \beta \in[1,2].
}
Then there exists $\beta(x,x_1) \in [1,2]$ such that
\Eq{*}{
h(x) = \frac{f_0(x)+r_{\beta(x,x_1)}(x)}{f_0(x_1)+r_{\beta(x,x_1)}(x_1)} \text{ for }x>x_1.
}
For $x_1\ge x_0$ and $x \ge x_0$ we have 
\Eq{*}{
\calF(x) \ni r_1(x) \le r_{\beta(x,x_1)}(x) \le r_2(x) \in \calF(x).
}
Thus $r_{\beta(x,x_1)}(x) \in \calF(x)$ (it means that $\calF$ is considered as variable of $x$). Similarly $r_{\beta(x,x_1)}(x_1) \in \calF(x_1)$.
%Also $r_{3/2}(x) \in \calF$.
Furthermore, by Taylor's theorem, $\frac{1}{p+\calF(x)}=\tfrac 1p +\calF(x)=\tfrac 1p+o(x-x_0)$. Thus, for $x>x_1$,
\Eq{*}{
h(x) &= \frac{f_0(x)+r_{\beta(x,x_1)}(x)}{f_0(x_1)+\calF(x_1)}=(f_0(x)+r_{\beta(x,x_1)}(x))\cdot(\frac{1}{f_0(x_1)}+o(x_1-x_0))\\
}
Recall that $x_1$ was fixed but arbitrary, so we can substitute $x_1 \leftarrow s$, where $s>x_0$. 
Furthermore we can consider $h_s(x):=f_0(s) \cdot h(x)$, as their generate the same quasi-arithetic mean.
Then, for $x>s>x_0$,
\Eq{*}{
h_s(x)&=f_0(s) \cdot (f_0(x)+r_{\beta(x,s)}(x))\cdot(\frac{1}{f_0(s)}+o(s-x_0))\\
&=(f_0(x)+r_{\beta(x,s)}(x))\cdot(1+f_0(s)\cdot o(s-x_0))\\
&=(f_0(x)+r_{\beta(x,s)}(x))\cdot (1+o(f_0(s) \cdot (s-x_0))).
}
Thus we get a family of functions $\mathcal{H}=\{h_{s}(x)\}_{s>x_0}$
\Eq{*}{
h_s(x)=(f_0(x)+r_{\beta(x,s)}(x))\cdot (1+o(f_0(s) \cdot (s-x_0))),\,x>s>x_0.
}
So we get
\Eq{*}{
h_s(x) &\ge (f_0(x)+r_{1}(x))\cdot (1+o(f_0(s) \cdot (s-x_0))),\,x>s>x_0;\\
h_s(x) &\le (f_0(x)+r_{2}(x))\cdot (1+o(f_0(s) \cdot (s-x_0))),\,x>s>x_0.
}
We can now pass $s \to x_0$ and obtain 
\Eq{*}{
h_{x_0}(x):=\lim_{s \to x_0}h_s(x) \in [f_0(x)+r_1(x),f_0(x)+r_2(x)],\,x>x_0.
}
Therefore $h_{x_0} \in f_0+\calF$. Furthermore, as $\QA{h}=\QA{h_s}$ for all $s>x_0$ then, applying \ref{CC:Pales}, $\QA{h_{x_0}}=\QA{h}$.

Finally we obtain $\QA{\tilde h}=\QA{h}=\QA{h_{x_0}} \in \QA{f_0+\calF}$.
\end{proof}

This theorem has a very useful corollary
\begin{cor}
\label{cor:onesided_vanishing}
 Let $I$ be an interval, $x_0 \in I$. The family of quasi-arithmetic means generated by right-(left-)sided differentiable function at $x_0$ with $f'_+(x_0)= 0$ ($f'_-(x_0)= 0$) is an interval-type set.
\end{cor}
\begin{proof}
Let $\QA{f} \le \QA{h} \le \QA{g}$. Suppose $f(x_0)=g(x_0)=1$. We know that 
\Eq{*}{
f(x)=1+o(x-x_0)\quad \text{ and }\quad g(x)=1+o(x-x_0) \quad \text{ for }x>x_0.
}
Take $f_0 \equiv 1$ and $\calF=o(x-x_0)$. Then the pair $f_0$, $\calF$ satisfies all conditions of Theorem~\ref{thm:Error}. 
Furthermore $f,g \in f_0+\calF$, therefore we get 
\Eq{*}{
h \in f_0+\calF=1+o(x-x_0) \text{ for }x>x_0.
}
It implies that $h'_+(x_0)$ exists and $h'_+(x_0)=0$.
\end{proof}

\subsection{Interval-type sets involving smoothness assumptions}
In the following section we are going to present some interval-type sets in a family of quasi-arithmetic means involving smoothness assumptions of their generating functions. Recall that all means are considered on a common interval $I$.

First result will concern existence of one-sided derivative at certain point
\begin{thm}
\label{thm:onesided_nonvanishing}
Let $x_0 \in \interior I$ and $f \colon I \to \R$ be a continuous and monotone functions which has a right-(left-)sided differentiable function at $x_0$ with $f'_{-}(x_0)\ne 0$ [$f'_+(x_0)\ne 0$]. If $\QA{g} \in [\QA{f}]^{\ast}$ for some $g \colon I \to \R$ then $g$ is right-(left-)sided differentiable at $x_0$ too and $g'_{-}(x_0)\ne 0$ [$g'_+(x_0)\ne 0$].
\end{thm}

\begin{proof}
By \eq{E:Eqcond} we may assume $f(x_0)=g(x_0)=0$. Then we have
\Eq{*}{
\lim_{x \to x_0^+} \frac{g(x)-g(x_0)}{x-x_0}
&=\lim_{y \to 0^+} \frac{g\circ f^{-1}(y)}{f^{-1}(y)-x_0}\\
&=\lim_{y \to 0^+} \frac{g\circ f^{-1}(y)}{y} \cdot \lim_{y \to 0^+} \frac{y}{f^{-1}(y)-x_0}\\
&=\lim_{y \to 0^+} \frac{g\circ f^{-1}(y)}{y} \cdot \lim_{x \to x_0^+} \frac{f(x)}{x-x_0}
}
But $g\circ f^{-1}(0)=g(x_0)=0$. Moreover, as $\QA{g}$ is comparable with $\QA{f}$, we know that $g \circ f^{-1}$ is either convex or concave (see \ref{CC:gf^-1} and \ref{CC:fg^-1}).
In particular there exists a one-side derivative $(g \circ f^{-1})_+(0)$. Moreover, as $g \circ f^{-1}$ is monotone and convex in some neighborhood of $0$, we get $(g \circ f^{-1})_+(x_0)\ne 0$. Furthermore $f'_+(x_0)$ exists and is nonzero.

Finally we obtain that there exists $g'_+(x_0)$ and
\Eq{*}{
g'_+(x_0)=(g\circ f^{-1})'_+(0) \cdot f'_+(x_0) \ne 0.
}
\end{proof}

Having this already proved we have an immediate corollary
\begin{cor}
\label{cor:onesided}
 Quasi-arithmetic means generated by functions which are right-(left-)sided differentiable functions at certain point $x_0 \in I$ with $f'_{-}(x_0)\ne 0$ [$f'_+(x_0)\ne 0$] is an interval. 
\end{cor}

This result can be somehow improved. Namely if derivative of $f'(x_0)$ and $g'(x_0)$ exists and is nonzero and \eq{E:sandwich} holds, then it is also the case in $h$. In can be formally expressed in term of the following 
\begin{thm}
\label{thm:dernon_van}
Quasi-arithmetic means generated by a functions differentiable at certain point $x_0 \in I$ with $f'(x_0)\ne 0$ is an interval.
\end{thm}

\begin{proof}
Let $f$, $g$, and $h$ be strictly increasing, and \eq{E:sandwich} holds. If $f,\,g$ are differentiable at $x_0$ and $f'(x_0)g'(x_0)\ne 0$ then, by Corollary~\ref{cor:onesided}, we know that both $h'_+(x_0)$ and $h'_-(x_0)$ exists and are nonzero.
Now, by \ref{CC:fg^-1}, $h$ is convex with respect to $f$. Thus $h'_-(x_0) \le h'_+(x_0)$. 

Similarly, by \ref{CC:gf^-1}, $h$ is concave with respect to $g$ and, consequently, $h'_-(x_0) \ge h'_+(x_0)$.
\end{proof}

Furthermore, this result could be rearrange in the case of continuous derivative

\begin{prop}
 \label{prop:C1ne0}
Let $I$ be an open interval. Quasi-arithmetic means generated by a functions belonging to $\mathcal{C}^1(I)$ with nowhere vanishing derivative is an interval.
\end{prop}

\begin{proof}
Suppose that $f,\,g,\,h$ are increasing, $f,\,g \in \mathcal{C}^1(I)$, $f' \cdot g' \ne 0$ and \eq{E:sandwich} holds. Then, by Theorem~\ref{thm:dernon_van}, $h$ is differentiable and $h' \ne 0$. Moreover, by \ref{CC:f'/g'-samemonotonicity}, we know that
$h'/f'$ is non-decreasing and $h'/g'$ is non-increasing.

Let $x_0\in I$. We can take affine transformations such that $f(x_0)=g(x_0)=h(x_0)=0$ and $f'(x_0)=g'(x_0)=h'(x_0)=1$.
Thus for all $x>x_0$, $f'(x)\le h'(x) \le g'(x)$. Similarly, for $x<x_0$, $g'(x)\le h'(x) \le f'(x)$.
Therefore
\Eq{*}{
l(x):=\min(f'(x),g'(x)) \le h'(x) \le \max(f'(x),g'(x)):=u(x) \text{ for all }x.
}
But $l(x_0)=u(x_0)=1$ and both $l$ and $u$ are continuous so $h'$ is continuous at the point $x_0$. 
But $x_0$ was arbitrary so $h \in \mathcal{C}^1(I)$.
\end{proof}

We are now heading toward one-sided differentiability in certain point (without vanishing or nonvanishing assumptions). First we will prove a very useful lemma.

\begin{lem}
\label{lem:comp0non0}
If $g,h \colon I \to \R$ are right-(left-)sided differentiable at $x_0 \in \interior I$ with $g_+'(x_0)=0$ and $h_+'(x_0) \ne 0$ ($g_-'(x_0)=0$ and $h_-'(x_0) \ne 0$) then $\QA{g}$ is not comparable with $\QA{h}$.
\end{lem}

\begin{proof}
Take $\varepsilon >0$ such that $x_0+\varepsilon \in I$. There exist affine transformations $\hat{g}$ and $\hat{h}$ of $g$ and $h$, respectively, such that
\begin{itemize}
 \item $\hat{g}(x_0)=\hat{h}(x_0)=0$,
 \item $\hat{h}(x_0+\varepsilon)=1$,
 \item $\hat{g}(x_0+\varepsilon)=2$.
 \end{itemize}
We obviously have $\hat{h}'(x_0)>0$ and $\hat{g}'(x_0)=0$.
It implies that $\hat g(x)<h(x)$ is some right neighborhood of $x_0$.
Let $\xi>x_0$ be the smallest number such that $\hat g(\xi)=h(\xi)$.
Now, as $\hat g(x)\le h(x)$ for all $x \in [x_0,\xi]$ we obtain
\Eq{*}{
\hat g^{-1}(y)\ge h^{-1}(y)\text{ for all } y \in [0,\hat g(\xi)].
}
in particular, for $y:=\hat g(\xi)/2$ we get
\Eq{*}{
\QA{g}_{1/2}(x_0,\xi)&=\QA{\hat g}_{1/2}(x_0,\xi)=\hat g^{-1}(\hat g(\xi)/2)=\hat g^{-1}(y) \\
&\ge  h^{-1}(y)=h^{-1}(h(\xi)/2)=\QA{h}_{1/2}(x_0,\xi)=\QA{h}_{1/2}(x_0,\xi)
}
To obtain the converse we can adopt this proof assuming $\varepsilon<0$ (regarding we will consider maximal $\xi$ and some sings will be changed).
\end{proof}

Having this already proved we can skip the assumption about nonvanishing one-sided derivative and obtain

\begin{thm}
The family of quasi-arithmetic means generated by functions differentiable at some point $x_0 \in \interior I$ is an interval-type set.
 \end{thm}

\begin{proof}
Suppose that $f$ and $g$ are differentiable at $x_0$ and $\QA{f} \le \QA{h} \le \QA{g}$. As $\QA{f}$ and $\QA{g}$ are comparable, by Lemma~\ref{lem:comp0non0} we get that either $f'(x_0)=g'(x_0)=0$ or $f'(x_0) \cdot g'(x_0)\ne 0$. In the first case we can use Corollary~\ref{cor:onesided_vanishing} to obtain differentiability of $h$ at $x_0$, while in the second case we use Theorem~\ref{thm:dernon_van}.
\end{proof}

In the case when $I$ is open we can take an intersection of these interval-type sets over all $x_0 \in I$ to obtain

\begin{cor}
The family of quasi-arithmetic means defined on an open interval $I$ generated by differentiable functions is an interval-type set.
 \end{cor}

%Applying Theorem~\ref{thm:Error} with $f_0(x)=x-x_0$ and $\calF = o(f_0)$
%\begin{theorem}
%Let $x_0 \in I$ and $\calF$ be a family of continuous functions defined of $I$ satisfying 
%\begin{itemize}
% \item $r(x_0)=0$ for all $r \in \calF$;
% \item if $r \in \calF$ then $\alpha \cdot r \in \calF$ for all $\alpha>0$;
% \item if $r_1,r_2 \in \calF$ and $r_1 \le r \le r_2$ then $r \in \calF$.
%\end{itemize}
%Then $\{f \colon f(x)-f(x_0) \in \calF\}$

%Suppose $x_1>x_0$ and $f(x_0)=g(x_0)=h(x_0)=0$, $f(x_1)=g(x_1)=h(x_1)=1$. 

%Then $f(x)\in\calF$ and $g(x)\in\calF$ for  $x>x_0$ (resp. $x<x_0$) [this assumption is resistant under affine transformation]. By $\QA{f} \le \QA{h} \le \QA{g}$ we get $f(x) \ge g(x) \ge h(x)$ for all $x \in (x_0,x_1)$. Consequently $h(x)\in \calF$ for $x>x_0$ (resp. $x<x_0$). 

\def\cprime{$'$} \def\R{\mathbb R} \def\Z{\mathbb Z} \def\Q{\mathbb Q}
  \def\C{\mathbb C}

\end{document}